\documentclass[12pt]{amsart}
\usepackage{amsthm,amsmath,amsxtra,amssymb,amsfonts,latexsym, mathrsfs, dsfont,bm}
\usepackage[initials]{amsrefs}
\usepackage[usenames]{color}
\usepackage[hyperindex,pagebackref]{hyperref}

\newtheorem{lemma}{Lemma}%[section]

\newtheorem{theorem}[lemma]{Theorem}

\newtheorem{rem}[lemma]{Remark}

\newcommand{\re}{\begin{rem}\rm}
  \newcommand{\mar}{\end{rem}}
\newtheorem{exam}[lemma]{Example}
\newcommand{\ex}{\begin{exam}\rm}
\newcommand{\amp}{\end{exam}}

\newcommand{\F}{\mathcal{F}}

\newcommand{\N}{\mathcal{N}}

\newcommand{\nz}{\mathbb{N}}
\newcommand{\ez}{\mathbb{E}}

\newcommand{\al}{\alpha}
\newcommand{\de}{\delta}
\newcommand{\si}{\sigma}
\newcommand{\ga}{\gamma}
\newcommand{\la}{\lambda}
\newcommand{\bt}{\beta}

\newcommand{\eps}{\varepsilon}
\newcommand{\8}{\infty}

\newcommand{\Prob}{\operatorname{Prob}}
\newcommand{\Var}{\operatorname{Var}}

\allowdisplaybreaks[1]
\title[Kolmogorov's LIL for noncommutative martingales]{Kolmogorov's law of the iterated logarithm for noncommutative martingales}
\author{Qiang Zeng}
\date{\today}
\address{Department of Mathematics, University of Illinois, Urbana, IL 61801}
\email{zeng8@illinois.edu}

\begin{document}

\begin{abstract}
We prove Kolmogorov's law of the iterated logarithm for noncommutative martingales. The commutative case was due to Stout. The key ingredient is an exponential inequality proved recently by Junge and the author.
\end{abstract}
\maketitle

\section{Introduction}
In probability theory, law of the iterated logarithm (LIL) is among the most important limit theorems and has been studied extensively in different contexts. The early contributions in this direction for independent increments were made by Khintchine, Kolmogorov, Hartman--Wintner, etc; see \cite{Bau} for more history of this subject. Stout generalized Kolmogorov and Hartman--Wintner's results to the martingale setting in \cite{Sto,Sto2}. The extension of LIL for independent sums in Banach spaces were due to Kuelbs, Ledoux, Talagrand, Pisier, etc; see \cite{LT91} and the references therein for more details in this direction. In the last decade, there has been new development for LIL results of dependent random variables; see \cite{Wu, ZW} and the references therein for more details. However, it seems that the LIL in noncommutative (= quantum) probability theory has only been proved recently by Konwerska \cite{Ko08, Ko} for Hartman--Wintner's version. Even the Kolmogorov's LIL for independent sums in the noncommutative setting is not known. The goal of this paper is to prove Kolmogorov's version of LIL for noncommutative martingales.

Let us first recall Kolmogorov's LIL. Let $(Y_n)_{n\in\nz}$ be an independent sequence of square-integrable, centered, real random variables. Put $S_n=\sum_{i=1}^n Y_i$ and $s_n^2=\Var(S_n)=\sum_{i=1}^n\ez(Y_i^2)$. Here and in the following $\ez$ denotes the expectation and $\Var$ denotes the variance. For any $x>0$, we define the notation $L(x)=\max\{1, \ln\ln x\}$. In 1929, Kolmogorov proved that if $s_n^2\to\8$ and
\begin{equation}\label{grow}
 |Y_n|\le \al_n\frac{s_n}{\sqrt{L(s_n^2)}} ~~a.s.
\end{equation}
for some positive sequence $(\al_n)$ such that $\lim_{n\to\8}\al_n=0$, then
\begin{equation}\label{kol0}
  \limsup_{n\to\8} \frac{S_n}{\sqrt{s_n^2 L(s_n^2)}} = \sqrt{2} ~~a.s..
\end{equation}
Later on, Hartman--Wintner \cite{HW} proved that if $(X_n)$ is an i.i.d. sequence of real, centered square-integrable random variables with variance $\Var(X_i)=\si^2$, then
\[
 \limsup_{n\to\8} \frac{S_n}{\sqrt{n L(n)}} = \sqrt{2}\si ~~a.s..
\]
de Acosta \cite{deA} simplified the proof of Hartman--Wintner. To compare the two results, if the sequence $(Y_n)$ are i.i.d. and uniformly bounded, then the two results coincide. Apparently, Hartman--Wintner's LIL does not contain Kolmogorov's version as a special case. However, Kolmogorov's LIL can be used in a truncation procedure to prove other LIL results; see, e.g., \cite{Sto2}.

Kolmogorov's LIL was generalized to martingales by Stout \cite{Sto}. Let $(X_n,\F_n)_{n\ge 1}$ be a martingale with $\ez(X_n)=0$. Let $Y_n=X_n-X_{n-1}$ for $n\ge 1, X_0=0$  be the associated martingale differences. Put $s_n^2=\sum_{i=1}^n \ez[Y_i^2|\F_{i-1}]$. Then Stout proved that if $s_n^2\to\8$ and \eqref{grow} holds, then $\limsup_{n\to\8} {X_n}/\sqrt{s_n^2L(s_n^2)} = \sqrt{2} ~a.s.$.

To state our main results, let us set up the noncommutative framework. Throughout this paper, we consider a noncommutative probability space $(\N,\tau)$. Here $\N$ is a finite von Neumann algebra and $\tau$ a normal faithful tracial state, i.e., $\tau(xy)=\tau(yx)$ for $x,y\in\N$. For $1\le p< \8$, define $\|x\|_p =[\tau(|x|^{p})]^{1/p}$ and $\|x\|_\8=\|x\|$ for $x\in\N$. In this paper $\|\cdot\|$ will always denote the operator norm. The noncommutative $L_p$ space $L_p(\N,\tau)$ (or $L_p(\N)$ for short) is the completion of $\N$ with respect to $\|\cdot\|_p$. $\tau$-measurable operators affiliated to $(\N,\tau)$ are also called noncommutative random variables; see \cite{FK, Ter} for more details on the measurability and noncommutative $L_p$ spaces. Let $(\N_k)_{k=1,2, \cdots }\subset \N$ be a filtration of von Neumann subalgebras with conditional expectation $E_k: \N\to \N_k$. Then $E_k(1)
=1$ and $E_k(axb)=aE_k(x)b$ for $a,b\in\N_k$ and $x\in\N$. It is well known that $E_k$ extends to contractions on $L_
p(\N,\tau)$ for $p\ge 1$; see \cite{JX03}.

Following \cite{Ko08}, a sequence $(x_n)$ of $\tau$-measurable operators is said to be almost uniformly bounded by a constant $K\ge 0$, denoted by $\limsup_{n\to\8}x_n\underset{a.u.}{\le} K$, if for any $\eps>0$ and any $\de>0$, there exists a projection $e$ with $\tau(1-e)<\eps$ such that
\begin{equation}\label{aub}
\limsup_{n\to\8}\|x_n e\| \le K+\de;
\end{equation}
and $(x_n)$ is said to be bilaterally almost uniformly bounded by a constant $K\ge 0$, denoted by $\limsup_{n\to\8}x_n\underset{b.a.u.}{\le} K$, if \eqref{aub} is replaced by
\[\limsup_{n\to\8}\|e x_n e\|\le K+\de.\]
Clearly, $\limsup_{n\to\8}x_n\underset{a.u.}{\le} K$ implies $\limsup_{n\to\8}x_n\underset{b.a.u.}{\le} K$.

For a $\tau$-measurable operator $x$ and $t>0$, the generalized singular numbers \cite{FK} are defined by
\[\mu_t(x)=\inf\{s>0: \tau(1_{(s,\infty)}(|x|))\le t\}.\]
In this paper, we use $1_A(a)$ to denote the spectral projection of an operator $a$ on the Borel set $A$. According to \cite{Ko08}, a sequence of operators $(x_i)$ is said to be uniformly bounded in distribution by an operator $y$ if there exists $K>0$ such that $\sup_i\mu_t(x_i)\le K\mu_{t/K}(y)$ for all $t>0$. Let $(x_n)$ be a sequence of mean zero self-adjoint independent random variables. Konwerska \cite{Ko} proved that if $(x_n)$ is uniformly bounded in distribution by a random variable $y$ such that $\tau(|y|^2)=\si^2<\8$, then
\[
\limsup_{n\to\infty}\frac{1}{\sqrt{nL(n)}} \sum_{i=1}^n x_i \underset{b.a.u.}{\le}C\si.
\]
Note that if the sequence $(x_n)$ is i.i.d., which is the case in the original version of Hartman--Wintner's LIL, then $(x_n)$ is uniformly bounded in distribution by $x_1$. Essentially, the condition of uniform boundedness in distribution requires the sequence to be almost identically distributed.

Our main result is an extension of Stout's result to the noncommutative setting. Let $(x_n)_{n\ge 0}$ be a noncommutative self-adjoint martingale with $x_0=0$ and $d_i=x_i-x_{i-1}$ the associated martingale differences. Define $s_n^2=\|\sum_{i=1}^n E_{i-1}(d_i^2)\|_\8$ and $u_n=[L(s_n^2)]^{1/2}$.
\begin{theorem}\label{koli}
 Let $0=x_0,x_1,x_2,\ldots$ be a self-adjoint martingale in $(\N,\tau)$. Suppose $s_n^2\to\8$ and $\|d_n\|_\8\le \al_n s_n/u_n$ for some sequence $(\al_i)$ of positive numbers such that $\al_n\to 0$ as $n\to\8$. Then
 \[\limsup_{n\to\infty}\frac{x_n}{s_n u_n}\underset{a.u.}{\le}2.\]
\end{theorem}
So far as we know, this is the first result on the LIL for noncommutative martingales. A natural question is to ask for the lower bound of LIL. As observed in \cite{Ko08}, however, one can only expect an upper bound for LIL in the general noncommutative setting. Indeed, consider a free sequence of semicircular random variables $(x_n)$ (the so-called free Gaussian random variables \cite{VDN}) such that the law of $x_n$ is $\ga_{0,2}$ (in notation, $x_n\sim \ga_{0,2}$) for all $n$. Here $\ga_{0,2}$ has density function $p(x)=\frac1{2\pi}\sqrt{4-x^2}$ for $-2\le x\le 2$. Then it is well known in free probability theory that
\[
\frac1{\sqrt{n}}\sum_{i=1}^n x_i \sim \ga_{0,2}.
\]
It follows that $\lim_{n\to\8}\sum_{i=1}^n x_i /\sqrt{n L(n)}=0$ in the norm topology since a random variable with law $\ga_{0,2}$ is bounded. Therefore there is no reasonable notion of the positive LIL lower bound for the free semicircular sequence. Comparing our LIL results with classical ones, we lose a constant of $\sqrt{2}$. However, since there is no hope to obtain an LIL lower bound in the general noncommutative theory, we are more interested in the order of the fluctuation for general noncommutative martingales. It is also commonly acknowledged that going from the commutative theory to the noncommutative setting usually requires considerably more technologies \cite{PX03}. Due to these reasons, it seems fair to have the constant 2 in the noncommutative martingale setting.

We will recall some preliminary facts in Section 2. The main result will be proved in Section 3. We will discuss some further questions in Section 4.

\section{Preliminaries}
In this section, we give some basic definitions and collect some preliminary facts. Let us recall the vector valued noncommutative $L_p$ spaces for $1\le p\le\8$ introduced by Pisier \cite{Pis} and Junge  \cite{Ju}. Let $(x_n)$ be a sequence in $L_p(\N)$ and define
\[
 \|(x_n)\|_{L_p(\ell_\8)}=\inf \{\|a\|_{2p}\|b\|_{2p}: x_n=ay_nb, \|y_n\|_\8\le 1\}.
\]
Then $L_p(\ell_\8)$ is defined to be the closure of all sequences with $\|(x_n)\|_{L_p(\ell_\8)}<\8$. It was shown in \cite{JX07} that if every $x_n$ is self-adjoint, then
\[
 \|(x_n)\|_{L_p(\ell_\8)}=\inf\{\|a\|_p: a\in L_p(\N), a\ge0, -a\le x_n\le a \text{ for all } n\in \nz\}.
\]
Similarly, Junge and Xu introduced in \cite{JX07} the space $L_p(\ell_\8^c)$ with norm
\begin{align*}
  &\|(x_i)_{i\in I}\|_{L_p(\ell^c_\8)} \\
  =~& \inf\{\|a\|_p: a\in L_p(\N), a\ge0, -a\le x_i^*x_i\le a \text{ for all } i\in I\}\\
  =~& \inf \{\|b\|_{p}: x_i=y_i b, \|y_i\|_\8\le 1 \mbox{ for all } i\in I\}.
\end{align*}
%They also defined $L_p(c_0)\subset L_p(\ell_\8)$ as the space of all $(x_n)\subset L_p(\N)$ such that there are $a,b\in L_{2p}(\N)$ and $(y_n)\subset \N$ satisfying
%\begin{equation}\label{lpc0}
% x_n=ay_nb \quad \mbox{and}\quad \lim_{n\to\8}\|y_n\|_\8 = 0.
%\end{equation}
%Konwerska \cite{Ko08} used the space $L_p(c)\subset L_p(\ell_\8)$, which is defined similar to $L_p(c_0)$ with \eqref{lpc0} replaced by
%\[
% x_n=ay_nb \quad \mbox{and}\quad \lim_{n\to\8}\|y_n-y_\8\|_\8 = 0 \mbox{ for some } y_\8\in \N.
%\]
%The norms of $L_p(c_0)$ and $L_p(c)$ are the same as $L_p(\ell_\8)$. Clearly $L_p(c_0)\subset L_p(c)$. It is easy to verify that all the spaces we mentioned above are Banach spaces.

The following result is the noncommutative asymmetric version of Doob's maximal inequality proved by Junge \cite{Ju}. We add a short proof to elaborate on the constant which is implicit in the original paper.
\begin{theorem}\label{doob}
 Let $4\le p\le \8$. Then, for any $x\in L_p(\N)$, there exists $b\in L_p(\N)$ and a sequence of contractions $(y_n)\subset\N$ such that
 \[
  \|b\|_p\le 2^{2/p}\|x\|_p \quad \mbox{and} \quad E_n x= y_n b, \text{ for all } n\ge0.
 \]
\end{theorem}
\begin{proof}
  This follows from \cite{Ju}*{Corollary 4.6}. Indeed, setting $r=p\ge 4$ and $q=\8$, we find $E_nx=az_n b$ for $a, z_n\in \N$ and $b\in L_p(\N)$. Let $y_n=az_n/\|az_n\|\in \N$ and $b'=\|az_n\|b\in L_p(\N)$. Then $(y_n)$ is a sequence of contractions, $E_nx=y_n b'$, and
  \[
  \|b'\|_p\le \|a\|_\8\|b\|_p\sup_n\|z_n\|_\8\le c(p,q,r)\|x\|_p,
  \]
  where $c(p,q,r)\le c^{1/2}_{q/(q-2)} c^{1/2}_{r/(r-2)} =c_1^{1/2} c^{1/2}_{p/(p-2)}$ and $c_p$ is the constant in the dual Doob's inequality. Note that $1\le p/(p-2)\le 2$. By Lemma 3.1 and Lemma 3.2 of \cite{Ju}, we find that $c_p\le 2^{2(p-1)/p}$ for $1\le p\le 2$. It follows that $c(p,q,r)\le 2^{2/p}$.
\end{proof}
Suppose $(x_i)_{m\le i\le n}$ is a martingale in $L_p(\N)$. According to Theorem \ref{doob}, there exist $b\in L_p(\N)$ and contractions $(y_i)_{m\le i\le n}\subset \N$ such that $x_i= y_i b$ for $m\le i\le n$ and $\|b\|_p\le 2^{2/p}\|x_n\|_p$ for $p\ge 4$. It follows that
\[
 \|(x_i)_{m\le i\le n}\|_{L_p(\ell_\8^c)} \le 2^{2/p}\|x_n\|_p.
\]
Doob's inequality will be used in this form in the proof of our main result.

Our proof of LIL for martingales relies on the following exponential inequality proved in \cite{JZ}. Its proof was based on Oliveira's approach to the matrix martingales \cite{Ol}.
\begin{lemma}\label{expin}
Let $(x_k)$ be a self-adjoint martingale with respect to the filtration $(\N_k, E_k)$ and $d_k=x_k-x_{k-1}$ be the associated martingale differences such that
\begin{enumerate}
\item[i)] $\tau(x_k)=x_0=0$; {\rm ii)} $\|d_k\|\le M$; {\rm iii)} $\sum_{k=1}^n E_{k-1}(d_k^2) \le D^2 1$.
\end{enumerate}
Then \[\tau(e^{\la x_n}) \le \exp[(1+\eps)\la^2 D^2]\] for all $\eps\in(0,1]$ and all $\la\in[0, \sqrt{\eps}/(M+M\eps)]$.
\end{lemma}

Another important tool in our proof is a noncommutative version of Borel--Cantelli lemma. To state this result, we recall from \cite{Ko08} that for a self-adjoint sequence $(x_i)_{i\in I}$ of random variables, the column version of tail probability is by definition
\begin{align*}
    &{\rm Prob}_c\Big(\sup_{i\in I} \|x_i\| > t\Big)\\
 = &\inf\{s>0: \exists \text{ a projection $e$ with } \tau(1-e)<s\\
 &\text{ and $\|x_i e\|_\8\le t$ for all $i\in I$} \}
\end{align*}
for $t>0$. It is immediate that
\begin{equation}\label{ord1}
  {\rm Prob}_c(\sup_{i\in I} \|x_i\| > t)\le {\rm Prob}_c(\sup_{i\in I} \|x_i\| > r)
\end{equation}
for $t\ge r$ and that if $a_i\ge 1$ for $i\in I$, then
\begin{equation}\label{ord2}
  {\rm Prob}_c(\sup_{i\in I} \|x_i\| > t)\le {\rm Prob}_c(\sup_{i\in I} \|a_i x_i\| > t).
\end{equation}
Using the notation $\Prob_c$, we state two lemmas which are taken from \cite{Ko08}.
\begin{lemma}
[Noncommutative Borel--Cantelli lemma]\label{nbcl}
Let $\cup_n I_n=\{n\in\nz: n\ge n_0\}$ for some $n_0\in \nz$ and $(z_n)$ be a sequence of self-adjoint random variables. If for any $\delta>0$,
\[\sum_{n\ge n_0}{\rm Prob}_c\Big(\sup_{m\in I_n} \| z_m\| > \ga +\delta\Big)<\infty,\] then
\[\limsup_{n\to \8}  z_n \underset{a.u.}{\le} \ga.\]
\end{lemma}
\begin{lemma}[Noncommutative Chebyshev inequality]\label{cheb}
Let $(x_i)_{i\in I}$ be a self-adjoint sequence of random variables. For $t > 0$ and $1\le p <\8$,
 $$\Prob_c(\sup_n \|x_n\| > t )\le t^{-p}\|x\|^p_{L_p(\ell_\8^c)}.$$
\end{lemma}

\section{Law of the iterated logarithm}
According to \cite{Bau}, the original proof of Kolmogorov's LIL is comparably expensive as that of Hartman--Wintner. However, our proof of Kolmogorov's LIL here seems to be relatively easier than (the upper bound of) Hartman--Wintner's version for the commutative case due to the exponential inequality (Lemma \ref{expin}).
\begin{proof}[Proof of Theorem \ref{koli}]
 Let $\eta\in(1,2)$ be a constant which we will determine later. To avoid annoying subscripts, we write $s(k_i)=s_{k_i}$ in the following. Using the stopping rule in \cite{Sto}, we define $k_0=0$ and for $n\ge 1$,
 \[
  k_n=\inf\{j\in\nz: s_{j+1}^2\ge \eta^{2n}\}.
 \]
Then $s_{k_n+1}^2\ge \eta^{2n}$ and $s_{k_n}^2 < \eta^{2n}$. Note that given $\eps'>0$ there exists $N_1(\eps')>0$ such that for $n>N_1(\eps')$,
\begin{align*}
  &s^2_{k_{n}+1}u^2_{k_{n}+1}/(s(k_{n+1})^2u(k_{n+1})^2)\\
  \ge~& \eta^{-2}\ln \ln \eta^{2n}/\ln\ln \eta^{2(n+1)}\ge (1-\eps')^2\eta^{-2}.
\end{align*}
Then $s_mu_m\ge (1-\eps')\eta^{-1} s(k_{n+1}) u(k_{n+1})$ for $k_n< m\le k_{n+1}$. For any $\de'>0$, we can find $\de,\eps'>0$ and $\eta\in (1,2)$ such that $1+\de'>\eta(1+\de)(1-\eps')^{-1}$. Fix $\bt>0$ which will be determined later. Using the notation ${\rm Prob}_c$ with order relations \eqref{ord1} and \eqref{ord2}, we have for $n>N_1(\eps')$
\begin{equation}\label{redu}
  \begin{aligned}
 &{\rm Prob}_c\Big(\sup_{k_n< m \le k_{n+1}} \Big\|\frac{x_m }{s_m u_m}\Big\| > \beta (1+\delta')\Big) \\
 \le~ & {\rm Prob}_c\Big(\sup_{k_n< m \le k_{n+1}} \Big\|\frac{\la x_m}{s(k_{n+1}) u(k_{n+1})}\Big\|> \la \beta (1+\delta)\Big).
 \end{aligned}
\end{equation}
By Lemma \ref{cheb} and Theorem \ref{doob}, we have for $p\ge4$,
\begin{align*}
 &{\rm Prob}_c\Big(\sup_{k_n < m \le k_{n+1}} \Big\|\frac{\la x_m}{s(k_{n+1}) u(k_{n+1})}\Big\|> \la \beta (1+\delta)\Big)\\
\le~& (\la \beta (1+\delta))^{-p}\Big\|\Big(\frac{\la x_m}{s(k_{n+1}) u(k_{n+1})}\Big)_{k_n < m \le k_{n+1}}\Big\|^p_{L_p(\ell^c_{\8})}\\
\le ~&(\la \beta (1+\delta))^{-p}\big(2^{2/p}\big)^p\Big\|\frac{\la x(k_{n+1})}{s(k_{n+1}) u(k_{n+1})}\Big\|^p_p.
\end{align*}
Using the elementary inequality $|u|^p\le p^pe^{-p}(e^u+e^{-u})$, functional calculus and Lemma \ref{expin} with $M=\al(k_{n+1})s(k_{n+1})/{u(k_{n+1})}$, $D^2=s(k_{n+1})^2$, we find
\begin{align*}
 &\Big\|\frac{\la x(k_{n+1})}{s(k_{n+1}) u(k_{n+1})}\Big\|^p_p\\
 \le~& p^pe^{-p}\tau \left(\exp\Big(\frac{\la x(k_{n+1})}{s(k_{n+1}) u(k_{n+1})}\Big) + \exp\Big(-\frac{\la x(k_{n+1})}{s(k_{n+1}) u(k_{n+1})}\Big) \right) \\
 \le ~& 2\Big(\frac{p}{e}\Big)^p \exp\Big(\frac{(1+\eps)\la^2}{u(k_{n+1})^2}\Big)
\end{align*}
provided $0\le \la \le \frac{\sqrt{\eps}u(k_{n+1})^2}{(1+\eps)\al(k_{n+1})}$ and $0<\eps\le1$. Hence we obtain
\begin{align*}
 &{\rm Prob}_c\Big(\sup_{k_n< m \le k_{n+1}} \Big\|\frac{\la x_m}{s(k_{n+1}) u(k_{n+1})} \Big\|> \la \beta (1+\delta)\Big) \\
\le ~&8\left(\frac{ p}{\la \bt (1+\de)e}\right)^p\exp\left(\frac{(1+\eps)\la^2}{u(k_{n+1})^2}\right).
\end{align*}
Now optimizing in $p$ gives $p=\la\bt(1+\de)$ and thus,
\begin{align*}
 &\Prob_c\Big(\sup_{k_n < m \le k_{n+1}}\Big\| \frac{\la x_m}{s(k_{n+1}) u(k_{n+1})} \Big\|> \la \beta (1+\delta)\Big)\\
 \le ~& 8\exp\Big(\frac{(1+\eps)\la^2}{u(k_{n+1})^2}-\bt(1+\de)\la\Big).
\end{align*}
Put $\la = \bt(1+\de)u(k_{n+1})^2/(2(1+\eps))$. Since $\al_n\to 0$, for any $\eps>0$, there exists $N_2>0$ such that for $n>N_2$, $0<\al(k_{n+1})\le \frac{2\sqrt{\eps}}{\bt(1+\de)}$, which ensures that we can apply Lemma \ref{expin}. This also implies $p\ge4$ for large $n$. It follows that
\[
 \Prob_c\Big(\sup_{k_n < m \le k_{n+1}} \Big\|\frac{\la x_m}{s(k_{n+1}) u(k_{n+1})}\Big\|> \la \beta (1+\delta)\Big) \le (\ln s(k_{n+1})^2)^{-\frac{\bt^2(1+\de)^2}{4(1+\eps)}}.
\]
Notice that $s(k_{n+1})^2\ge s(k_n+1)^2\ge \eta^{2n}$. Setting $\bt=2$ in the beginning of the proof, we have
\[
 \Prob_c\Big(\sup_{k_n < m \le k_{n+1}}\Big\| \frac{\la x_m}{s(k_{n+1}) u(k_{n+1})}\Big\|> \la \beta (1+\delta)\Big) \le [(2\ln \eta) n]^{-\frac{(1+\de)^2}{1+\eps}}.
\]
By choosing $\eps$ small enough so that $(1+\de)^2/(1+\eps)>1$, we find that for $n_0=\max\{N_1,N_2\}$,
\[
 \sum_{n\ge n_0} \Prob_c\Big(\sup_{k_n < m \le k_{n+1}}\Big\| \frac{\la x_m}{s(k_{n+1}) u(k_{n+1})}\Big\|> \la \beta (1+\delta)\Big) <\8.
\]
Then \eqref{redu} and Lemma \ref{nbcl} give the desired result.
\end{proof}
\section{Further questions}
Without the growing condition on martingale differences $d_n$, Stout proved Hartman--Wintner's LIL in \cite{Sto2} under the additional assumption that the martingale differences are stationary ergodic. At the time of this writing, it is still not clear to us whether a ``genuine'' version (i.e., it does not satisfy Kolmogorov's growing condition) of Hartman--Wintner's LIL is possible for noncommutative martingales. It would be interesting to see such a result in the future.
\section*{Acknowledgement}
The author would like to thank Marius Junge for inspiring discussions, Tianyi Zheng for the help on relevant literature, and Ma{\l}gorzata Konwerska for sending him the preprint \cite{Ko}. He is also grateful to the anonymous referee for careful reading and suggestions on improving the paper. 

\bibliographystyle{plain}
\bibliography{lil_ref}
\end{document}